\documentclass[11pt]{amsart}
\usepackage[all]{xy}
\usepackage{amscd}
\usepackage{amsfonts}
\usepackage{amsmath}
\usepackage{amssymb}
\usepackage{amsthm}
\usepackage{bm}
\usepackage{fancyhdr}
\usepackage{graphics}
\usepackage{graphicx}
\usepackage{latexsym}
\usepackage{mathrsfs}
\usepackage{parskip}
\usepackage{slashed}
\usepackage[usenames]{color}
\allowdisplaybreaks

\newtheorem{thm}{Theorem}[section]
\newtheorem{cor}[thm]{Corollary}
\newtheorem{lem}[thm]{Lemma}
\newtheorem{prop}[thm]{Proposition}

\theoremstyle{definition}
\newtheorem{defn}{Definition}[section]
\theoremstyle{remark}

\theoremstyle{example}

\numberwithin{equation}{section}

\newcommand{\tr}{\mathrm{tr}}

\newcommand{\mn}{\sqrt{-1}}

\newcommand{\dbar}{\overline{\partial}}
\newcommand{\ddbar}{\sqrt{-1}\partial\overline{\partial}}
\newcommand{\md}{\mathrm{d}}

\begin{document}

\title{Higher dimensional generalizations of twistor spaces}
\makeatletter
\let\uppercasenonmath\@gobble
\let\MakeUppercase\relax
\let\scshape\relax
\makeatother
\author{Hai Lin${}^{1}$, Tao Zheng${}^{2}$ }
\address{${}^{1}$ Yau Mathematical Sciences Center, Tsinghua University, Beijing 100084, P. R. China}
\email{hailin@mail.tsinghua.edu.cn}
\address{${}^{2}$ Institut Fourier, Universit\'{e} Grenoble Alpes, 100 rue des maths,
Gi\`{e}res 38610, France}
\email{zhengtao08@amss.ac.cn}

\begin{abstract}
We construct a generalization of twistor spaces of hypercomplex manifolds and hyper-K\"{a}hler manifolds $M$, by generalizing the twistor $\mathbb{P}^{1}$ to a more general complex manifold $Q$. The resulting manifold $X$ is complex if and only if $Q$ admits a holomorphic map to $\mathbb{P}^1$. We make
branched double covers of these manifolds. Some class of these branched double covers can give rise to non-K\"{a}hler Calabi-Yau manifolds. We show that these manifolds $X$ and their branched double covers are non-K\"{a}hler. In the cases that $Q$ is a balanced manifold, the resulting manifold $X$ and its special branched double cover have balanced Hermitian metrics.
\end{abstract}
\maketitle
\section{\textbf{Introduction}}

\label{sec_introduction}

Non-K\"{a}hler geometries exist in both heterotic string theory and type II
string theory, in the presence of fluxes. In the compactification of
heterotic string theory to four dimensional Minkowski spacetime \cite%
{Candelas:1985en, Witten:1985bz, Strominger:1986uh}, the internal
six-manifolds can become non-K\"{a}hler in the presence of fluxes. Various
models of constructing heterotic manifolds and their vector-bundles have
been put forward, see for example \cite{Strominger:1986uh,Becker:2006et,Fu:2006vj,Fu:2008ga,Becker:2009df,Andreas:2010cv,Melnikov:2014ywa,Anderson:2014xha,delaOssa:2014lma,Anderson:2015iia}.
They play an important role in searching for realistic
string theory vacua with four dimensional Minkowski spacetime. The non-K\"{a}%
hler manifolds and balanced manifolds can also occur in type II string theory,
in the presence of three-form fluxes and five-brane sources. For example,
they have appeared in the context of eight-dimensional Hermitian manifolds in type IIB string theory, see \cite{Lin:2016lwe, Minasian:2016txd, Prins:2013koa, Rosa:2013lwa}.

An interesting type of non-K\"{a}hler manifolds are the balanced Hermitian manifolds (see \cite{Mic82}%
). They are Hermitian manifolds with a Hermitian form $\omega $ and a
holomorphic form. For a non-K\"{a}hler balanced manifold, its Hermitian form
$\omega $ is not closed, however, $\omega^{p-1}$ is closed, where $p$ is
the complex dimension of the manifold. Under appropriate blowing-downs or
contractions of curves, some classes of non-K\"{a}hler balanced manifolds
can become K\"{a}hler and have projective models in algebraic geometry (see
for example \cite{Poon86, Lebrun Poon, Lin:2014lya}).

Twistor spaces \cite{Penrose:1967wn, AHS78} provide an important type of non-K\"{a}hler manifolds. Given an
oriented Riemannian four-manifold $M$, there is an associated twistor space of $M$, sometimes denoted as $\mathrm{Tw}(M)$. The construction of the twistor spaces uses a special twistor $\mathbb{P}^{1}$. The twistor $\mathbb{P}^{1}$ parametrizes the set of almost complex structures of the associated twistor space \cite{AHS78}.

There are several classes of manifolds whose twistor spaces are of great
interest. One class of manifolds are the four-manifolds with self-dual
conformal structure \cite{Poon86,Ped Poon,AHS78}, such as the connected sum
of $n$ copies of $\mathbb{P}^{2}$s, i.e., $n\mathbb{P}^{2}$. Their twistor
spaces are complex (see for example \cite{Lebrun Poon,Poon,Lebrun}). For the
simplest cases $n=0$ and $1$ respectively, the manifolds are $S^{4}$ and $%
\mathbb{P}^{2}$ respectively, and their twistor spaces are $\mathbb{P}^{3}$ and
the flag manifold $\mathbb{F}_{1,2}$ respectively, which are K\"{a}hler \cite%
{Hitchin}. For the cases $n\geq 2$, the twistor spaces $\mathrm{Tw}(n\mathbb{P}%
^{2})$ are non-K\"{a}hler manifolds (see \cite{Lebrun Poon,Poon,Lebrun}).
The branched double covers of the twistor spaces were analyzed in \cite%
{Lin:2014lya,Heckman:2013sfa}, in which by choosing appropriate
branch divisors, the branched double covers can give rise to non-K\"{a}hler
Calabi-Yau manifolds, with trivial canonical bundle. They provide
interesting examples of non-K\"{a}hler Calabi-Yau manifolds (see for example
\cite{Tosatti, Tseng:2012wca, Lin:2014lya}). We can also
construct various vector bundles on them \cite{Lin:2014lya}.

Another class of manifolds whose twistor spaces are very interesting are the
hyper-K\"{a}hler and hypercomplex manifolds. Hypercomplex manifolds are
complex manifolds with a two-sphere's worth of complex structures satisfying quaternionic relations (see for
example \cite{Boyer, Spindel etal, Joyce, Salamon}). They are
generalizations of the hyper-K\"{a}hler manifolds, but their associated
Hermitian forms are not closed. The hyper-K\"{a}hler manifold is a special
case of the hypercomplex manifold when it has a K\"{a}hler structure. Their
twistor spaces \cite{Salamon01,KV98} are also complex manifolds, equipped with balanced metrics
(see for example \cite{Kaledin98, KV98, Tomberg}). There are also twistor
spaces of various other manifolds, see for example \cite{Alekseevsky
etal,PP,Howe:1996kj} and the references therein.

In this paper, we construct a generalization of the twistor spaces of
hypercomplex manifolds and hyper-K\"{a}hler manifolds, by generalizing the
twistor $\mathbb{P}^{1}$ to a more general complex manifold $Q$, to obtain a
higher dimensional manifold $X$. We will also make branched double coverings of
these manifolds $X$, branched along appropriate divisors. Some of these
branched double covers produce non-K\"{a}hler Calabi-Yau manifolds. In the
cases that $Q$ is a balanced manifold, the resulting manifold $X$ and its special
branched double cover have balanced metrics.

The organization of this paper is as follows. In Section \ref{sec_ordinary
twistor space}, we give a preliminary account of the hypercomplex and hyper-K\"{a}hler manifolds, and their twistor spaces. In Section \ref{sec_higher
dimension construction}, we present a generalization of the twistor spaces of
hypercomplex manifolds and hyper-K\"{a}hler manifolds $M$, by changing the
twistor $\mathbb{P}^{1}$ to a general complex manifold $Q$. Then in Section \ref{sec_higher dimension construction branch cover}, we make branched
double covers of these manifolds $X$, branching along appropriate divisors.
Some of these branched double covers can generate non-K\"{a}hler Calabi-Yau
manifolds. In Sections \ref{sec_a basic lemma} and \ref{sec_positivity}, we describe a useful Lemma and discuss positivity methods that will be useful for showing the Hermitian metrics of these manifolds that we construct. In Sections \ref{sec_balanced metric_higher_analog_hyper-Kahler} and \ref{sec_balanced metric_higher_analog_hypercomplex}, we show
that the resulting manifold $X$ and its special branched double cover have balanced Hermitian metrics, if in addition $Q$ is a balanced manifold. In Section \ref{sec_non-Kahlerity}, we show that these manifolds $X$ and their branched double covers are non-K\"{a}hler. Finally in Section \ref{sec_discussion} we
briefly make conclusions and discuss some related directions.

\noindent {\bf Acknowledgements}
We would like to thank T. Fei, J. Heckman, I. Melnikov, R. Minasian, A. Patel, D.
Prins, D. Rosa, J. Simon, V. Tosatti, B. Wu, S.-T. Yau, and J. Zhang for
discussions and communications. The work of HL was supported in part by NSF grant
DMS-1159412, NSF grant PHY-0937443 and NSF grant DMS-0804454, and in part by
YMSC, Tsinghua University. The work of TZ was supported by National Natural Science Foundation of China grant No. 11401023, and his post-doc position is supported by the ERC grant ALKAGE. The authors are also grateful to the anonymous referees and the editor for their careful reading and helpful suggestions which greatly improved the paper.

\section{\textbf{Higher dimensional analogs of twistor spaces}}
\label{sec_higher dimension}
\subsection{Twistor space of hyper-K\"ahler and hypercomplex manifolds}
\label{sec_ordinary twistor space}

Let us first give preliminary accounts of hyper-K\"{a}hler and hypercomplex manifolds and their twistor spaces.

Let $(M,\,\mathbf{I},\mathbf{J},\mathbf{K},g)$ be a
hypercomplex manifold with $\dim_{\mathbb{C}}M=r=2k$, where $\mathbf{I},\mathbf{J},\mathbf{K}$ are the
complex structures: $TM\longrightarrow TM$,~with relations
\begin{equation*}
\mathbf{I}^{2}=\mathbf{J}^{2}=\mathbf{K}^{2}=-1,\;\mathbf{I}\mathbf{J}=
\mathbf{K},\;\mathbf{J}\mathbf{K}=\mathbf{I},\;\mathbf{K}\mathbf{I}=\mathbf{J}
\end{equation*}
and $g$ is the Riemannian metric compatible with $\mathbf{I},\mathbf{J},
\mathbf{K}$. Then we can define real two-forms
\begin{align*}
\omega _{\mathbf{I}}(X,Y):=&g(\mathbf{I}X,\,Y),\quad \omega _{\mathbf{J}}(X,Y):=g(
\mathbf{J}X,\,Y),\\
 \omega _{\mathbf{K}}(X,Y):=&g(\mathbf{K}X,\,Y),\quad
\forall \;X,\,Y\in \mathfrak{X}(M),
\end{align*}
where $\mathfrak{X}(M)$ is the space of vector fields on $M$. Let
\begin{equation*}
e_{1},\cdots ,e_{k},\,\mathbf{I}e_{1},\cdots ,\mathbf{I}e_{k},\,\mathbf{J}
e_{1},\cdots ,\mathbf{J}e_{k},\,\mathbf{K}e_{1},\cdots ,\mathbf{K}e_{k}
\end{equation*}
be the orthonormal basis for $M$. Then
\begin{equation*}
e^{1},\cdots ,e^{k},\,-\mathbf{I}e^{1},\cdots ,-\mathbf{I}e^{k},\,-\mathbf{J}
e^{1},\cdots ,-\mathbf{J}e^{k},\,-\mathbf{K}e^{1},\cdots ,-\mathbf{K}e^{k}
\end{equation*}
is the dual basis. Here we remark that, for example, $\mathbf{I}\mathbf{J}e^1=-\mathbf{K}e^1$ by the extended definition $(\mathbf{I}e^i)(e_k)=e^i(\mathbf{I}e_k)$. Hence if we consider
\begin{equation*}
e_{1},\cdots ,e_{k},\,\mathbf{J}e_{1},\cdots ,\mathbf{J}e_{k},\,\mathbf{I}
e_{1},\cdots ,\mathbf{I}e_{k},\,\mathbf{K}e_{1},\cdots ,\mathbf{K}e_{k}
\end{equation*}
\begin{equation*}
e^{1},\cdots ,e^{k},\,-\mathbf{J}e^{1},\cdots ,-\mathbf{J}e^{k},\,-\mathbf{I}
e^{1},\cdots ,-\mathbf{I}e^{k},\,-\mathbf{K}e^{1},\cdots ,-\mathbf{K}e^{k}
\end{equation*}
and define
\begin{equation*}
\alpha ^{i}=e^{i}-\sqrt{-1}\mathbf{I}e^{i},\quad \eta ^{i}=-\mathbf{J}e^{i}-
\sqrt{-1}\mathbf{K}e^{i},\quad i=1,\cdots k,
\end{equation*}
then we can deduce
\begin{equation}\label{omegai}
\omega _{\mathbf{I}}=\frac{\sqrt{-1}}{2}\sum\limits_{i=1}^{k}\left( \alpha
^{i}\wedge \overline{\alpha }^{i}+\eta ^{i}\wedge \overline{\eta }%
^{i}\right).
\end{equation}
Similar to this process and using some simple calculation, we can also get
\begin{align}
\label{omegaj}
\omega _{\mathbf{J}}=& \frac{1}{2}\sum\limits_{i=1}^{k}
\left( \alpha^{i}\wedge \eta ^{i}+\overline{\alpha }^{i}\wedge \overline{\eta }^{i}\right) , \\
\label{omegak}
\omega _{\mathbf{K}}=& \frac{\sqrt{-1}}{2}\sum\limits_{i=1}^{k}
\left(-\alpha ^{i}\wedge \eta ^{i}+\overline{\alpha }^{i}\wedge \overline{\eta}^{i}\right).
\end{align}
Let us parametrize the twistor $\mathbb{P}^{1}$ with $[Z_{1},Z_{2}]$, where $Z_{1},Z_{2}$
are complex numbers. Then the standard biholomorphic
map between $\mathbb{P}^1$ and $\mathbb{S}^2$ is defined by
\begin{align}
\varsigma :& \,\mathbb{P}^{1}\longrightarrow \mathbb{S}^{2}  \nonumber \\
& [Z_{1},Z_{2}]\longmapsto \left( \frac{Z_{1}\overline{Z}_{2}+\overline{Z}
_{1}Z_{2}}{|Z_{1}|^{2}+|Z_{2}|^{2}},\sqrt{-1}\frac{\overline{Z}
_{1}Z_{2}-Z_{1}\overline{Z}_{2}}{|Z_{1}|^{2}+|Z_{2}|^{2}},
\frac{|Z_{1}|^{2}-|Z_{2}|^{2}}{|Z_{1}|^{2}+|Z_{2}|^{2}}\right).\nonumber
\end{align}
In this paper, considering the extended definition of the complex structures on forms above, we use the map
\begin{align}
\tilde{\varsigma}:& \,\mathbb{P}^{1}\longrightarrow \mathbb{S}^{2}  \notag \\
& [Z_{1},Z_{2}]\longmapsto \left( \frac{|Z_{2}|^{2}-|Z_{1}|^{2}}{
|Z_{1}|^{2}+|Z_{2}|^{2}},\frac{Z_{1}\overline{Z}_{2}+\overline{Z}_{1}Z_{2}}{
|Z_{1}|^{2}+|Z_{2}|^{2}},\sqrt{-1}\frac{\overline{Z}_{1}Z_{2}-Z_{1}\overline{
Z}_{2}}{|Z_{1}|^{2}+|Z_{2}|^{2}}\right) .
\end{align}
Therefore, our orientation on $\mathbb{S}^{2}$ is opposite with the ordinary
one.

We can consider the twistor space $Z$ of hypercomplex manifold $M$ with
a product smooth structure $Z=M\times \mathbb{P}^{1}$ (see for example \cite{Hitchin:1986ea}). However, the complex structure on $Z$ is given by
\begin{equation}\label{complexstructure1}
\overline{\mathbf{I}}=\left( \frac{|Z_{2}|^{2}-|Z_{1}|^{2}}{
|Z_{1}|^{2}+|Z_{2}|^{2}}\mathbf{I}+\frac{Z_{1}\overline{Z}_{2}+\overline{Z}_{1}Z_{2}}{
|Z_{1}|^{2}+|Z_{2}|^{2}}\mathbf{J}+\sqrt{-1}\frac{\overline{Z}_{1}Z_{2}-Z_{1}\overline{
Z}_{2}}{|Z_{1}|^{2}+|Z_{2}|^{2}}\mathbf{K},\,I_{0}\right) .
\end{equation}
where $I_{0}$ is the standard complex structure on $\mathbb{P}^{1}$.

\subsection{Higher dimensional analogs of the twistor space of hyper-K\"{a}%
hler and hypercomplex manifolds}

\label{sec_higher dimension construction}

We can generalize the twistor space of hypercomplex manifolds $M$ to higher dimensional analogs, by changing the twistor $\mathbb{P}^{1}$ to a more general and usually higher dimensional complex manifold $Q$. In this subsection we present the construction of this generalization.

Let $Q$ be a complex manifold with $\dim_{\mathbb{C}}Q=n$ and
\begin{equation}
\label{map_01}
h:Q\longrightarrow \mathbb{P}^{1}
\end{equation}
be a smooth map. Assume that $(U,z^1,\cdots,z^n)$ and $(V,\zeta=Z_1/Z_2)$ are the local coordinates of $Q$ and $\mathbb{P}^1$ respectively. The map $h$ can be expressed as
$$
h:\; z=(z^1,\cdots,z^n)\mapsto\zeta=\zeta(z).
$$
Using \eqref{complexstructure1} and the map $\tilde{\varsigma}\circ h:\;Q\longrightarrow \mathbb{S}^{2}$, we
can define an almost complex structure on the manifold $X=M\times Q$ with product smooth structure by (for later use, we use local coordinates to express it)
\begin{equation}
\underline{\mathbf{I}}:=\left( \mathbf{I}_M,\,I\right)
\end{equation}
where $I$ is the complex structure on $Q$, and
$$
\mathbf{I}_M:=\frac{1-\zeta \overline{\zeta }}{1+|\zeta |^{2}}(z)\mathbf{I}+\frac{
\zeta +\overline{\zeta }}{1+|\zeta |^{2}}(z)\mathbf{J}+\sqrt{-1}\frac{\overline{\zeta }
-\zeta }{1+|\zeta |^{2}}(z)\mathbf{K}.
$$
\begin{thm}
Using the notations above, $(X,\,\underline{\mathbf{I}})$ is a complex
manifold if and only if $h$ is a holomorphic map.
\end{thm}
\begin{proof}
We use the Newlander-Nirenberg theorem \cite{Newlander Nirenberg}.
This theorem says that complex coordinates exist if for any $(1,0)$
form $\theta $, i.e., $\theta $ is a complex-valued one-form with $\underline{
\mathbf{I}}\theta =\sqrt{-1}\theta $, one has
\begin{equation*}
\mathrm{d}\theta =\theta ^{i}\wedge \beta ^{i},
\end{equation*}%
for $(1,0)$ forms $\theta ^{i}$ and general one-forms $\beta ^{i}$. This can be seen as
the complex version of the Frobenius integrability condition. For any $(1,0)$ form $\varphi$ on $M$ for the complex
structure $\mathbf{I}$ on $M$, this can also be seen as a one-form on $X$, and it follows that (cf. \cite{Hitchin:1986ea})(Note the remarks about the complex structure on forms above)
\begin{equation*}
\mathbf{I}_M(\varphi-\zeta \mathbf{K}\varphi)=\mn (\varphi-\zeta \mathbf{K}\varphi).
\end{equation*}
Therefore, let $\varphi
^{1},\cdots ,\varphi ^{r}$ be a local basis of $(1,0)$ forms for the complex
structure $\mathbf{I}$ on $M$.
Then
\begin{equation*}
\varphi ^{i}-\zeta \mathbf{K}\varphi ^{i},\quad \mathrm{d}z^{j},\quad 1\leq
i\leq r,\quad 1\leq j\leq n
\end{equation*}
give a basis for the $(1,0)$ forms of $X$.

Write a $(1,0)$ form $\theta$ for the complex structure $\underline{\mathbf{I}}$ as   $\theta
=\varphi -\zeta \mathbf{K}\varphi $, where $\varphi $ is a $(1,0)$ form for
the complex structure $\mathbf{I}$. Then  we have
\begin{equation*}
\mathrm{d}\theta=\mathrm{d}_M(\varphi -\zeta \mathbf{K}\varphi )-\partial_Q\zeta\wedge\mathbf{K}\varphi-\overline{\partial}_Q\zeta\wedge\mathbf{K}\varphi.
\end{equation*}
Obviously, the Nirenberg tensor of $\mathbf{I}_M$ is zero, which implies $\mathrm{d}_M(\varphi -\zeta \mathbf{K}\varphi )\in \Lambda^{2,0}M\oplus\Lambda^{1,1}M$, where $\Lambda^{p,q}M$ is for the complex structure $\mathbf{I}_M$. Then we get
\begin{equation*}
\mathrm{d}\theta =\mathrm{d}(\varphi -\zeta \mathbf{K}\varphi )\equiv
-\overline{\partial }_{Q}\zeta \wedge \mathbf{K}\varphi \quad (\text{mod}
\;\varphi ^{1}-\zeta \mathbf{K}\varphi ^{1},\cdots ,\varphi ^{r}-\zeta
\mathbf{K}\varphi ^{r},\,\mathrm{d}z^{1},\cdots ,\mathrm{d}z^{n}).
\end{equation*}
Therefore, $\underline{\mathbf{I}}$ is integrable if and only if
\begin{equation}
\overline{\partial }_{Q}\zeta =0,
\end{equation}%
i.e., $h$ is a holomorphic map.
\end{proof}
Grauert and Remmert \cite{grauertremmert} proved that for any proper holomorphic map $f:\,M\longrightarrow N$ between complex spaces $M$ and $N$, if $A\subset M$ is a subvariety, then $f(A)\subset N$ is also a subvariety. Since the subvarieties of $\mathbb{P}^1$ are discrete points or $\mathbb{P}^1$ itself, in the case where $Q$ is compact,  $h$ is a surjective holomorphic map or local constant map. We can also prove this conclusion directly. Indeed, if $h$ is not surjective and $p\in \mathbb{P}^1\setminus h(Q)$, then since $\mathbb{P}\setminus\{p\}$ is biholomorphic to $\mathbb{C}$, we can consider $h$ as a holomorphic function on $Q$, which is a local  constant function by the maximum principle.

By the definition of $\underline{\mathbf{I}}$, we can deduce the following obvious properties.
\begin{prop}
Using the notations as above, assume that $h:\,Q\longrightarrow \mathbb{P}^1$ is holomorphic. For the higher dimensional analogs $X$ of twistor spaces, there hold
\begin{enumerate}
\item The canonical projection $\pi:\, X\longrightarrow Q$ given by $(w,q)\mapsto q$ is holomorphic.
\item  For any $p\in M$ fixed, $i_p:\,Q\longrightarrow X$ defined by $q\mapsto (p,q)$ is also holomorphic.
\end{enumerate}
\end{prop}
We remark that for the case $Q=\mathbb{P}^1$, this property can be found in \cite{Kaledin98} and $i_p$ is the twistor line or twistor $\mathbb{P}^1$ corresponding to the point $p\in M$.

The holomorphic map $h:\,Q\longrightarrow \mathbb{P}^1$ is non-constant if and only if there exists a holomorphic line bundle $L$ on $Q$ such that we can find $s_1,s_2\in H^{0}(Q,L)$ with no common zero points. Indeed, for the ``if" direction, we can  define
 $$
 h:\;Q\longrightarrow \mathbb{P}^1,\;q\mapsto [s_1(q),s_2(q)].
 $$
It is easy to see that $h$ is well defined.  For the ``only if" direction, note that $h$ can be written locally as
$$
h|_{U_{\alpha}}:\;U_{\alpha}\longrightarrow \mathbb{P}^1,\quad q\mapsto [f_{1\alpha}(q),\,f_{2\alpha}(q)]\in \mathbb{P}^1,
$$
where $f_{i\alpha}:\;U_{\alpha}\longrightarrow \mathbb{C},\;i=1,2$ are holomorphic functions with no common zero points by the definition of $h$.
On $U_{\alpha}\cap U_{\beta}$, we have $[f_{1\alpha}(q),\,f_{2\alpha}(q)]=[f_{1\beta}(q),\,f_{2\beta}(q)]$ and hence there exists
a holomorphic function $h_{\alpha\beta}:\;U_{\alpha}\cap U_{\beta}\longrightarrow \mathbb{C}^{\ast}$ such that $(f_{1\alpha}(q),\,f_{2\alpha}(q))=h_{\alpha\beta}(q)(f_{1\beta}(q),\,f_{2\beta}(q))$, for all $q\in U_{\alpha}\cap U_{\beta}$. It is easy to check the cocycle condition that $h_{\alpha\beta}h_{\beta\gamma}=h_{\alpha\gamma}$ and $h_{\alpha\alpha}=1$. Hence $\{h_{\alpha\beta}\}$ define a holomorphic line bundle $L$ and $\{f_{i\alpha}\}\in H^0(Q,L),\,i=1,2$ have no common zero points.

In the case $\dim_{\mathbb{C}}Q=1$, this condition can be always satisfied since by the Riemann-Roch theorem, there exists non-constant meromorphic functions on any Riemann surfaces which are equivalent to holomorphic maps to $\mathbb{P}^1$.
Also all such maps are the branch covering of $\mathbb{P}^1$. This case was analysed
by \cite{Fei:2015kua}, and see related discussions of $g=3~$case \cite%
{Lin:2014lya} and $g=1$ case \cite{Heckman:2013sfa}.

As a result,  there must exist a meromorphic function on $Q$, i.e., $Q$ must have positive algebraic dimension.
However, in the case $\dim_{\mathbb{C}}Q\geq 2$, we do not have a similar simple conclusion as in the $n=1$ case since meromorphic functions can not always define holomorphic maps to $\mathbb{P}^1$.

However, we can give some examples of $Q$ with $\dim_{\mathbb{C}}Q\geq 2$. We first consider compact examples of $Q$. Assume that $\tilde{\pi}:\,E\longrightarrow \mathbb{P}^1$ is a holomorphic vector bundle with rank $r+1$. We get a projective bundle $\mathbb{P}(E)$ associated to $E$ which is the quotient of $E$ minus zero section by the natural action of $\mathbb{C}^{\ast}$. It is well-known that $\mathbb{P}(E)$ is a compact K\"{a}hler manifold and $\pi:\,\mathbb{P}(E)\longrightarrow \mathbb{P}^1$ induced from $\tilde{\pi}$ is a holomorphic map (see \cite[Proposition 3.18, Remark 3.19]{voisin1}). Let $Y\subset \mathbb{P}(E)$ be a complex sub-manifold ($Y$ can be discrete points). Then we can obtain a blowup $\widetilde{\mathbb{P}(E)}_{Y}$ of the projective bundle $\mathbb{P}(E)$. It is easy to deduce that $\widetilde{\mathbb{P}(E)}_{Y}$ is a compact K\"{a}hler manifold and the blowup map $\tau:\; \widetilde{\mathbb{P}(E)}_{Y}\longrightarrow \mathbb{P}(E)$ is a holomorphic map (\cite[Proposition 3.24]{voisin1}). Hence, we can take $Q$ as $\mathbb{P}(E)$ or  $\widetilde{\mathbb{P}(E)}_{Y}$. We can obtain a series of manifolds by these methods which can be chosen as $Q$. Here we just mention one of the simplest examples, the $n$-th Hirzebruch surface $\Sigma_n$ which is of the form $\mathbb{P}(\mathbb{C}\oplus \mathcal{O}_{\mathbb{P}^1}(n))$, where $\mathcal{O}_{\mathbb{P}^1}(n)$ is the holomorphic line bundle on $\mathbb{P}^1$. For example, $\Sigma_0$ is $\mathbb{P}^1\times\mathbb{P}^1$, $\Sigma_1$ is the blowup of $\mathbb{P}^2$ on one point, and for $n\geq 2$, $\Sigma_n$ can be obtained by desingularizing the cone in $\mathbb{P}^{n+1}$ over a rational normal curve spanning $\mathbb{P}^n$ and more details about the Hirzebruch surface can be found in \cite{jinfuqumian, gh}.

We can also take $Q$ as a non-compact complex manifold. For this aim, we first consider a natural holomorphic line bundle $\mathcal{O}_{\mathbb{P}^r}(-1)$ over $\mathbb{P}^r$, whose fiber at $\Lambda\in \mathbb{P}^r$ is the vector subspace $\Lambda\subset \mathbb{C}^{r+1}$ with rank $1$. We denote by $\mathcal{O}_{\mathbb{P}^r}(1)$ the dual of   $\mathcal{O}_{\mathbb{P}^r}(-1)$. For any $k\in \mathbb{Z}$, we define $\mathcal{O}_{\mathbb{P}^r}(k)=\left(\mathcal{O}_{\mathbb{P}^r}(\mathrm{sgn}k)\right)^{\otimes |k|}$. Note that $\mathcal{O}_{\mathbb{P}^r}(0)$ is the trivial line bundle on $\mathbb{P}^r$. Then, we can construct  a natural relative line bundle on $\mathbb{P}(E)$ as follows. Let $\mathcal{O}_{\mathbb{P}(E)}(-1)$ be the line subbundle of $\pi^{\ast}E$ over $\mathbb{P}(E)$ whose fiber at a point $(x,\Delta\subset E_x)$ is vector subspace $\Delta\subset E$ with rank $1$. We define $\mathcal{O}_{\mathbb{P}(E)}(1)$ as the dual of $\mathcal{O}_{\mathbb{P}(E)}(-1)$. The restriction of $\mathcal{O}_{\mathbb{P}(E)}(1)$ on each fiber of $\pi$ isomorphic to $\mathbb{P}^r$ is naturally isomorphic to $\mathcal{O}_{\mathbb{P}^r}(1)$. For any integer $k\in \mathbb{Z}$, we can define $\mathcal{O}_{\mathbb{P}(E)}(k):=\left(\mathcal{O}_{\mathbb{P}(E)}(\mathrm{sgn}k)\right)^{\otimes |k|}$. Again $\mathcal{O}_{\mathbb{P}(E)}(0)$ is the trivial line bundle on $\mathbb{P}(E)$. Then we can take $Q$ as  $\mathcal{O}_{\mathbb{P}(E)}(k)$ or its blowup  along its compact complex sub-manifolds.

For the non-compact case, we can also take $Q=\mathbb{C}^n$.  For this case there are some Picard (type) theorems. For any non-constant holomorphic map $h:\,\mathbb{C}\longrightarrow \mathbb{P}^1$, the Picard theorem states that $\mathbb{P}^1\setminus h(\mathbb{C})$ contains at most $2$ points. In the case of $n\geq 2$, for any non-constant holomorphic map $h:\,\mathbb{C}^n\longrightarrow \mathbb{P}^1$, there exists $(z_2,\cdots,z_n)\in \mathbb{C}^{n-1}$ such shat $$
h(\cdot,z_2,\cdots, z_n):\,\mathbb{C}\longrightarrow \mathbb{P}^1
$$
is a non-constant holomorphic map. This yields that $\mathbb{P}^1\setminus h(\mathbb{C}^n)$ also contains at most $2$ points. More details about this aspect can be found in \cite{griffiths} and the references therein. Therefore, we can also take $Q$ as $\mathbb{C}^n$.

\subsection{Branched double covers and non-K\"{a}hler Calabi-Yau manifolds}

\label{sec_higher dimension construction branch cover}

Using the holomorphic map $h:Q\longrightarrow \mathbb{P}^{1}$, we construct a
complex manifold $(X,\underline{\mathbf{I}})$ for the hypercomplex manifold
$M$. Let us consider a branched double cover of $X$ constructed in the previous subsection to obtain $\bar{X}$.
By
choosing appropriate divisors for the branch locus, the resulting double
covers can have trivial canonical bundle and hence they provide examples of
non-K\"{a}hler Calabi-Yau manifolds.

We construct branched double covers of $X$ by branching along a divisor $D\subset X$. Therefore we
define a double covering map
\begin{equation*}
\phi :\bar{X}\rightarrow X,
\end{equation*}%
branched along $D$. Such a double cover exists provided that $\mathcal{O}_X(D)=L^{\otimes 2}$ for some holomorphic line bundle $L$, where $\mathcal{O}_X(D)$ is the line bundle defined by the divisor $D$. The canonical class of $\bar{X}$ is given by
\begin{equation}
K_{\bar{X}}=\phi ^{\ast }\left( K_{X}\otimes L\right).
\label{junformula}
\end{equation}%
We can have different branched covers, depending on different types of
branch divisors, similar to those of \cite{Lin:2014lya,Heckman:2013sfa}.
There are several interesting types, particularly the divisors in the linear systems $|-mK_{X}|$ with $m=1$ or $m=2$ respectively.
One type is by choosing the divisor class $[D]=-K_{X}$ (see for example \cite{Heckman:2013sfa}). Another type is
by choosing the divisor class $[D]=-2K_{X}$ (see for example \cite%
{Lin:2014lya}), which produces non-K\"{a}hler Calabi-Yau manifolds
(manifolds with trivial canonical bundle but are nevertheless not K\"{a}%
hler) since $K_{\bar{X}}$ is trivial by the above adjunction formula \eqref{junformula}.

Since the non-K\"{a}hler Calabi-Yau manifolds and K\"{a}hler Calabi-Yau manifolds can be connected by a sequence of blowing downs and blowing ups
(see for example \cite{Poon86,Lebrun Poon,Poon,Lin:2014lya}), the non-K\"{a}hler Calabi-Yau manifolds play important roles in understanding the moduli
space of Calabi-Yau manifolds (see for example \cite{Reid}). Also note that we can construct various vector bundles on non-K\"{a}hler Calabi-Yau
manifolds, see for example \cite{Lin:2014lya}. For more discussions on non-K\"{a}hler Calabi-Yau manifolds, see for example \cite{Tosatti, Tseng:2012wca}
and references therein.

\section{\textbf{Balanced metrics on higher dimensional analogs of twistor spaces and their branched covers}}

\label{sec_higher dimension balanced metric}

In this section, we consider the balanced metrics and non-K\"ahlerity of the higher dimensional analogs and their branched covers.

\subsection{Exterior differentials on higher dimensional analogs of twistor spaces}
\label{sec_a basic lemma}
Let $(X,\,\underline{\mathbf{I}})$ be the higher dimensional analogs of twistor spaces for a hypercomplex manifold  $(M,\,\mathbf{I},\mathbf{J},\mathbf{K},g)$ with $\dim_{\mathbb{C}}M=r$  and a holomorphic map $$h:\;Q\longrightarrow \mathbb{P}^1.$$ Then we define
\begin{align}\label{omegam}
\omega_M:=g(\mathbf{I}_M(\cdot),\,\cdot)
=\frac{1-\zeta \overline{\zeta }}{1+|\zeta |^{2}}(z)\omega_{\mathbf{I}}+\frac{
\zeta +\overline{\zeta }}{1+|\zeta |^{2}}(z)\omega_{\mathbf{J}}+\sqrt{-1}\frac{\overline{\zeta }
-\zeta }{1+|\zeta |^{2}}(z)\omega_{\mathbf{K}},
\end{align}
and
$$
\sigma_M=\frac{-2\zeta }{\left( 1+|\zeta |^{2}\right) ^{2}}(z)\omega_{\mathbf{I}}+\frac{1-\zeta
^{2}}{\left( 1+|\zeta |^{2}\right) ^{2}}(z)\omega_{\mathbf{J}}+\sqrt{-1}\frac{1+\zeta ^{2}}{
\left( 1+|\zeta |^{2}\right) ^{2}}(z)\omega_{\mathbf{K}}.
$$
We have some basic properties of $\omega_M$ slightly different from \cite{Tomberg} as follows.
\begin{lem}
\label{basiclem}
Using the notations above, we have
\begin{align}
\overline{\partial }_{Q}\omega _{M}& \in \wedge ^{0,1}Q\otimes \wedge
^{2,0}M,  \label{0120q} \\
\partial _{Q}\omega _{M}& \in \wedge ^{1,0}Q\otimes \wedge ^{0,2}M,
\label{1002q} \\
\sqrt{-1}\partial _{Q}\overline{\partial }_{Q}\omega _{M}& =-2\left(h^{\ast
}\omega _{\mathbb{P}^{1}}\right)\wedge \omega _{M}.  \label{1111q}
\end{align}
In particular, for $r>2$, we have
\begin{align}
 \label{11n1q}
 \sqrt{-1}\partial _{Q}\omega _{M}\wedge \overline{\partial }_{Q}\omega
_{M}\wedge \omega _{M}^{r-3}& =\frac{4}{r-2}\left(h^{\ast }\omega _{\mathbb{P}
^{1}}\right)\wedge \omega _{M}^{r-1},  \\
\label{fq}
\sqrt{-1}\partial \overline{\partial }\omega _{M}^{r-1}& =\sqrt{-1}\partial
_{M}\overline{\partial }_{M}\omega _{M}^{r-1}+2(r-1)\left(h^{\ast }\omega _{
\mathbb{P}^{1}}\right)\wedge \omega _{M}^{r-1}.
\end{align}
Here $\omega _{\mathbb{P}^{1}}$ is the Fubini-Study metric on $\mathbb{P}%
^{1} $
\begin{equation*}
\omega _{\mathbb{P}^{1}}=\frac{\sqrt{-1}\mathrm{d}\zeta \wedge \mathrm{d}%
\overline{\zeta }}{\left( 1+|\zeta |^{2}\right) ^{2}}.
\end{equation*}
\end{lem}
\begin{proof}
Since $\zeta$ is a holomorphic map of $z$, the chain rule implies
\begin{align*}
 &\overline{\partial}_{Q}\left( \frac{1-\zeta \overline{\zeta }}{1+|\zeta |^{2}}(z)
,\,\frac{\zeta +\overline{\zeta }}{1+|\zeta |^{2}}(z),\,\sqrt{-1}\frac{
\overline{\zeta }-\zeta }{1+|\zeta |^{2}}(z)\right) \\
=& \partial_{\overline{\zeta}}\left( \frac{1-\zeta \overline{\zeta }}{1+|\zeta |^{2}}
,\,\frac{\zeta +\overline{\zeta }}{1+|\zeta |^{2}},\,\sqrt{-1}\frac{
\overline{\zeta }-\zeta }{1+|\zeta |^{2}}\right)(z)\overline{\partial}_{Q}\overline{\zeta} \\
=& \left( \frac{-2\zeta }{\left( 1+|\zeta |^{2}\right) ^{2}}(z),\,\frac{1-\zeta
^{2}}{\left( 1+|\zeta |^{2}\right) ^{2}}(z),\,\sqrt{-1}\frac{1+\zeta ^{2}}{
\left( 1+|\zeta |^{2}\right) ^{2}}(z)\right) \overline{\partial}_{Q}\overline{\zeta},
\end{align*}
i.e., $\overline{\partial}_Q\omega_M=\overline{\partial}_{Q}\overline{\zeta}\wedge\sigma_M$. For any vector fields $U,\,V\in \mathfrak{X}(M)$, a direct computation gives
$$
\sigma_{M}(U+\mn\mathbf{I}_MU,\,V)=0,
$$
which implies $\sigma_M\in \Lambda^{2,0}M$. Hence \eqref{0120q} holds and taking conjugation implies \eqref{1002q}.

As for \eqref{11n1q}, since $\mathbb{P}^1$ is a homogeneous manifold, we can take $(m,z)$ such that $(m,\zeta(z))=(m,0)$. Then using \eqref{omegai}, \eqref{omegaj}, \eqref{omegak} and \eqref{omegam}, a little complicated and direct computation gives \eqref{11n1q} (cf.\cite{Tomberg}).

Furthermore, we have
\begin{align}
&\mn \partial _{Q}\overline{\partial }_{Q}\left( \frac{1-\zeta\overline{\zeta}}{
1+|\zeta|^{2}}(z),\,\frac{\zeta+\overline{\zeta}}{1+|\zeta|^{2}}(z),\,\sqrt{-1}\frac{
\overline{\zeta}-\zeta}{1+|\zeta|^{2}}(z)\right)  \nonumber \\
=&\frac{\partial^2}{\partial \zeta\partial\overline{\zeta}}\left( \frac{-2\zeta}{\left( 1+|\zeta|^{2}\right) ^{2}
},\,\frac{1-\zeta^{2}}{\left( 1+|\zeta|^{2}\right) ^{2}},\,\sqrt{-1}\frac{
1+\zeta^{2}}{\left( 1+|\zeta|^{2}\right) ^{2}}\right)(z)\mn\partial_Q\zeta\wedge\overline{\partial_Q}\overline{\zeta}
\nonumber \\
=& -\sqrt{-1}\frac{2}{\left( 1+|\zeta|^{2}\right) ^{2}}(z)\left( \frac{1-\zeta
\overline{\zeta}}{1+|\zeta|^{2}}(z),\,\frac{\zeta+\overline{\zeta}}{1+|\zeta|^{2}}(z),\,
\sqrt{-1}\frac{\overline{\zeta}-\zeta}{1+|\zeta|^{2}}(z)\right)
\mn\partial_Q\zeta\wedge\overline{\partial_Q}\overline{\zeta},\nonumber
\end{align}
i.e., $\mn \partial _{Q}\overline{\partial }_{Q}\omega_M=-2\left(h^{\ast}\omega_{\mathbb{P}^1}\right)\wedge\omega_M$ and hence \eqref{1111q} holds.

From \eqref{0120q}, we have
\begin{equation*}
\partial _{M}\overline{\partial }_{Q}\omega _{M}\in \ \wedge
^{(0,1)}Q\otimes \wedge ^{(3,0)}M,
\end{equation*}
which implies
\begin{equation}
\label{f1q}
\partial _{M}\overline{\partial }_{Q}\omega _{M}\wedge \omega _{M}^{n-2}\in
\ \wedge ^{(0,1)}Q\otimes \wedge ^{(r+1,r-2)}M=\{0\}.
\end{equation}
Similarly, we have
\begin{equation}
\label{f2q}
\overline{\partial }_{M}\partial _{Q}\omega _{M}\wedge \omega _{M}^{r-2}\in
\ \wedge ^{(1,0)}Q\otimes \wedge ^{(r-2,r+1)}M=\{0\}.
\end{equation}
For $r>2$, at the same time, we can deduce that
\begin{equation}
\label{f3q}
\overline{\partial }_{Q}\omega _{M}\wedge \partial _{M}\omega _{M}\wedge
\omega _{M}^{r-3}\in \ \wedge ^{(0,1)}Q\otimes \wedge ^{(r+1,r-2)}M=\{0\}.
\end{equation}
Similarly, we can obtain
\begin{equation}
\label{f4q}
\partial _{Q}\omega _{M}\wedge \overline{\partial }_{M}\omega _{M}\wedge
\omega _{M}^{r-3}\in \ \wedge ^{(1,0)}Q\otimes \wedge ^{(r-2,r+1)}M=\{0\}.
\end{equation}
Since
\begin{equation*}
\sqrt{-1}\partial \overline{\partial }\omega _{M}^{r-1}=(r-1)\sqrt{-1}
\partial \overline{\partial }\omega _{M}\wedge \omega _{M}^{r-2}+\sqrt{-1}
(r-1)(r-2)\partial \omega _{M}\wedge \overline{\partial }\omega _{M}\wedge
\omega _{M}^{r-3}
\end{equation*}
and
\begin{align*}
\partial \omega _{M}=& \partial _{Q}\omega _{M}+\partial _{M}\omega _{M}, \\
\overline{\partial }\omega _{M}=& \overline{\partial }_{Q}\omega _{M}+
\overline{\partial }_{M}\omega _{M},
\end{align*}
using \eqref{1111q}, \eqref{11n1q}, \eqref{f1q}, \eqref{f2q}, \eqref{f3q}  and \eqref{f4q},
we can deduce \eqref{fq}.
\end{proof}
\subsection{Positivity and balanced metrics}
\label{sec_positivity}
In this subsection we discuss the methods of positivity and their relations to balanced metrics.

Assume that $Q$ is a complex manifold with $\dim_{\mathbb{C}}Q=n$. The basic concepts of positivity can be found in for example \cite[Chapter III]{demaillybook1}. A $(p,p)$ form $\varphi$ is said to be positive if for any $(1,0)$ forms $\gamma_j,\,1\leq j\leq n-p$, then
$$
\varphi\wedge\mn\gamma_1\wedge\overline{\gamma_1}\wedge\cdots\wedge\mn \gamma_{n-p}\wedge\overline{\gamma_{n-p}}
$$
is a positive $(n,n)$ form. Any positive $(p,p)$ form $\varphi$ is real, i.e., $\overline{\varphi}=\varphi$. In particular, in the local coordinates, a real $(1,1)$ form
\begin{align}\label{11}
\phi=\mn\phi_{i\overline{j}}\md z^i\wedge\md\overline{z}^j
\end{align}
is positive if and only if $(\phi_{i\overline{j}})$ is a semi-positive Hermitian matrix and we denote $\det \phi:=\det(\phi_{i\overline{j}})$. Similarly, a real $(n-1,n-1)$ form
\begin{align}\label{n-1}
\psi=&(\mn)^{n-1}\sum\limits_{i,j=1}^n(-1)^{\frac{n(n+1)}{2}+i+j+1}\psi^{\overline{j}i}\\
&\md z^1\wedge\cdots\wedge\widehat{\md z^i}\wedge\cdots\wedge\md z^n\wedge \md \overline{z}^1\wedge\cdots\wedge\widehat{\md\overline{z}^j}\wedge\cdots\wedge\md \overline{z}^n\nonumber
\end{align}
is positive if and only if $(\psi^{\overline{j}i})$ is a semi-positive Hermitian matrix and we denote $\det\psi:=\det(\psi^{\overline{j}i})$.
We remark that for $(1,1)$ and $(n-1,n-1)$ forms one also has the
stronger notion of positive definiteness, which is to require that the
Hermitian matrix $(\phi_{i \overline{j}})$ (resp. $(\psi^{\overline{j} i})$) is positive
definite. In this paper, we need this stronger notion and have the following lemma.
\begin{lem}[Michelsohn \cite{Mic82}]
\label{11n1n1}
Let $Q$ be a  complex manifold with $\dim_{\mathbb{C}}Q=n$. Then there exists a bijection from the space of positive definite $(1,1)$ forms to positive definite $(n-1,n-1)$  forms, given by
\begin{align}
\phi\mapsto \frac{\phi^{n-1}}{(n-1)!}.
\end{align}
\end{lem}
\begin{proof}
For a positive $(1,1)$ form $\phi$ defined as in \eqref{11}, we can deduce a positive $(n-1,n-1)$ form
\begin{align}
\frac{\phi^{n-1}}{(n-1)!}=&(\mn)^{n-1}\sum\limits_{k,\ell=1}^n(-1)^{\frac{n(n+1)}{2}+k+\ell+1}\mathrm{det}(\phi_{i\overline{j}})\tilde{\phi}^{\overline{\ell}k} \nonumber\\
&\md z^{ 1}\wedge\cdots\wedge\widehat{\md z^k}\wedge \cdots\wedge\md z^{n}\wedge\md\overline{z}^{ 1}\wedge\cdots\wedge\widehat{\md \overline{z}^{\ell}}\wedge\cdots\wedge\cdots\wedge \md\overline{z}^{n}\nonumber
\end{align}
where $(\tilde{\phi}^{\overline{\ell}k} )$ is the inverse matrix of $(\phi_{i\overline{j}})$, i.e.,  $\sum\limits_{\ell=1}^n\tilde{\phi}^{\overline{\ell}j}\phi_{k\overline{\ell}}=\delta_{k}^j$.

On the other hand, given a positive $(n-1,n-1)$ form $\psi$ defined as in \eqref{n-1}, there is a positive $(1,1)$ form
\begin{equation}\label{n-1n-111formula}
\xi=\mn\left(\mathrm{det}(\psi^{\overline{j}i})\right)^{\frac{1}{n-1}}\tilde{\psi}_{k\overline{\ell}}\md z^i\wedge\md\overline{z}^j
\end{equation}
such that
$$
\frac{\xi^{n-1}}{(n-1)!}=\psi,
$$
where $(\tilde{\psi}_{k\overline{\ell}} )$ is the inverse matrix of $(\psi^{\overline{j}i})$, i.e.,  $\sum\limits_{\ell=1}^n\psi^{\overline{\ell}j}\tilde{\psi}_{k\overline{\ell}}=\delta_{k}^j$.
\end{proof}
We remark that the above bijection can be found in \cite{Mic82} (cf. \cite{Tomberg}) and proved by orthonormal basis. Our proof here gives the explicit formulae involved.

Assume that $M$ is  complex manifold with $\dim_{\mathbb{C}}M=r$ and $\vartheta$ is a   positive $(1,1)$ form.  In local coordinates, it can be written as
\begin{align}\label{N11}
\vartheta=\mn\sum\limits_{i,j=1}^r \vartheta_{i\overline{j}}\md z^i\wedge\md\overline{z}^j.
\end{align}
Now positive $(1,1)$ form $\phi$ on $Q$ defined in   \eqref{11} and  positive $(1,1)$ form $\vartheta$ on $M$ defined in   \eqref{N11} can be seen as real $(1,1)$ form on $M\times Q$. For any positive function $A,\,B\in C^{\infty}(M\times Q,\mathbb{R})$,  we can deduce that
\begin{align}
A\vartheta^r\wedge\phi^{n-1}+B\vartheta^{r-1}\wedge\phi^n
\end{align}
is positive $(n+r-1,\,n+r-1)$ form on $M\times Q$. Using \eqref{n-1n-111formula}, it is easy to deduce that
\begin{align}\label{n-1n-111formulayong}
\tilde \xi=((n-1)!r!A)^{-\frac{r-1}{n+r-1}}((r-1)!n!B)^{\frac{r}{n+r-1}}\phi+((n-1)!r!A)^{\frac{n}{n+r-1}}((r-1)!n!B)^{-\frac{n-1}{n+r-1}}\vartheta
\end{align}
satisfies
$$
\frac{\tilde\xi^{n+r-1}}{(n+r-1)!}=A\phi^{n-1}\wedge\vartheta^r+B\phi^n\wedge\vartheta^{r-1}.
$$
\begin{defn}
Let $P$ be a  complex manifold with $\dim_{\mathbb{C}}P=p$. Then a  positive $(1,1)$ form $\xi$ on $P$ is called balanced metric if $\mathrm{d}\xi^{p-1}=0$.
\end{defn}
Obviously, the K\"{a}hler metric is balanced. Gray and Hervella observed that on a compact complex manifold $(M,\,\omega)$ with $\dim_{\mathbb{C}}M\geq 3$,
the condition $\md \omega^{k}=0$ for some $2\leq k\leq n-2$ implies that $M$ is K\"{a}her, i.e., $\md\omega=0$.
Indeed, $\md \omega^{k}=0$ implies $\omega^{n-3}\wedge\md\omega=0$, i.e., $L^{n-3}(\md \omega)=0$, where $L$ is the Lefschetz operator defined as wedging by $\omega$. By the Lefschetz decomposition for Hermitian manifolds, it follows that $L^{n-3}:\;\Lambda^{3}M\longrightarrow\Lambda^{2n-3}M$ is bijection and hence $\md\omega=0$. We can also use the fact  $\overline{\partial}\omega\wedge\omega^{n-3}=0$ in this case, and use a routine computation to get $\ast\partial\omega=-\frac{\mn}{(n-3)!}\overline{\partial}\omega\wedge\omega^{n-3}=0$, as required.\footnote{The authors would like to thank Prof. Valentino Tosatti for explaining this point.} Therefore,  it is meaningful to consider the balanced metric on non-K\"{a}hler complex manifolds.

Alessandrini and Bassanelli \cite{AB1,AB2} proved that for a modification $f:\,\tilde M\longrightarrow M$, $\tilde M$ is balanced if and only if $M$ is balanced. Here modification is defined as follows. Let $M$ and $\tilde M$ be complex manifolds (not necessarily compact) with $\dim_{\mathbb{C}}\tilde M=\dim_{\mathbb{C}}M=n$. Then a proper modification $f:\,\tilde M\longrightarrow M$ is a proper holomorphic map such that for a suitable analytic set $Y\subset M$ with $\mathrm{codim}Y\geq 2$ (called the center), $E:=f^{-1}(Y)$ (called the exceptional set of the modification) is a hypersurface and $f|_{\tilde M\backslash E}:\,\tilde M\backslash E\longrightarrow M\backslash Y$ is biholomorphic.

Michelsohn \cite{Mic82} showed that a compact complex manifold is balanced if and only if there exists no non-zero positive current $L$ of degree $(1,1)$ such that $L$ is the $(1,1)$ component of a boundary, i.e., $L=\partial \overline{S}+\overline{\partial}S$ with $S$ degree of $(1,0)$.

By Lemma \ref{11n1n1}, to find a balanced metric, it is sufficient to obtain a $\mathrm{d}$-closed positive $(p-1,p-1)$ form. In the following part, we will use this lemma to construct balanced metrics and remark that in some special branched covering cases, the balanced condition can be preserved.
\subsection{Balanced metrics on higher analogs of twistor spaces of hyper-K\"ahler manifolds}
\label{sec_balanced metric_higher_analog_hyper-Kahler}
It is easy to get a balanced metric if the higher analogs of twistor spaces come from hyper-K\"ahler manifolds. We have \begin{thm}
Let $(X,\,\underline{\mathbf{I}})$ be the higher analog of the twistor space of a  hyper-K\"ahler manifold $(M,\,\mathbf{I},\mathbf{J},\mathbf{K},g)$ with $\dim_{\mathbb{C}}M=r$ and $h:\;Q\longrightarrow \mathbb{P}^1$, where $Q$ is a   K\"ahler manifold with K\"ahler form $\omega_Q$ and $\dim_{\mathbb{C}}Q=n$. Then $\omega_M+t\omega_Q$ is a balanced metric (not K\"{a}hlerian) on $X$, where $t$ is any positive constant.
\end{thm}
\begin{proof}
It is sufficient to prove $\mathrm{d}(\omega_M+t\omega_Q)^{n+r-1}=0$.
Note that
\begin{align}
\label{kahlercompu}
\left(\omega_M+t\omega_Q\right)^{n+r-2}
=&t^{n-2}\binom {n+r-2 }{n}\omega_Q^{n-2}\wedge\omega_{M}^{r}+t^{n-1}\binom {n+r-2}{ n-1 }\omega_Q^{n-1}\wedge\omega_{M}^{r-1}\\
&+t^n\binom {n+r-2}{ n-2 }\omega_Q^{n}\wedge\omega_{M}^{r-2}\nonumber,\\
\label{hk1}
\md \left(\omega_M+t\omega_Q\right)
=&\dbar_{Q}\omega_M+\partial_{Q}\omega_M\\
\in&\wedge^{0,1}Q\otimes\wedge^{2,0}M+ \wedge^{1,0}Q\otimes\wedge^{0,2}M,  \nonumber
\end{align}
where for \eqref{hk1} we use that fact that $\mathrm{d}_M\omega_M=\mathrm{d}\omega_Q=0$ and lemma \ref{basiclem}. Note that $\md \left(\omega_M+t\omega_Q\right)\neq0$.
Then from \eqref{kahlercompu} and \eqref{hk1}, we can deduce
\begin{equation*}
\mathrm{d}(\omega_M+t\omega_Q)^{n+r-1}=(n+r-1)(\omega_M+t\omega_Q)^{n+r-2}\wedge \md \left(\omega_M+t\omega_Q\right)=0,
\end{equation*}
as required.
\end{proof}
Note that, in this case, $M$ and $Q$ can be non-compact.

\subsection{Balanced metrics on higher analogs of twistor spaces of compact hypercomplex manifolds and their branched covers}
\label{sec_balanced metric_higher_analog_hypercomplex}

In this subsection, we show that there exist balanced Hermitian metrics on the higher analogs of twistor spaces of compact hypercomplex manifolds and their special branched covers. This case also includes compact hyper-K\"ahler manifolds. By methods of positivity we find their explicit Hermitian metrics.

We first introduce a useful lemma which is slightly different from \cite[Lemma 2]{Tomberg}.
\begin{lem}\label{lem2}
Let $P$ be a compact complex manifold with $\dim_{\mathbb{C}}P=p$ and $\Xi$ be a holomorphic vector bundle with rank $r$. Let $\Xi=E\oplus F$ be a decomposition of $\Xi$ and $H$ and $H'$ be Hermitian forms on $\Xi$. If $H$ restricted on $E$ is strictly positive and $H'$ restricted on $F$ is strictly positive with $E\subset\mathrm{Ker}H'$, i.e., $H'(E,\cdot)=0$, then there exists a positive number $c$ such that $H+cH'$ is strictly positive on $\Xi$.
\end{lem}
\begin{proof}
We use the ideas from \cite{Tomberg}.
Since the manifold is compact, we just need to prove the conclusion locally. Let $e_1,\cdots,e_{s}$ be the local holomorphic frame basis on $E$ and $e_{s+1},\cdots e_r$ be the local holomorphic frame basis on $F$ such that
\begin{align}
H(e_i,e_j)=\delta_{ij}, \quad H'(e_{\alpha},e_{\beta})=\delta_{\alpha\beta},\quad 1\leq i,\;j\leq s,\quad s+1\leq \alpha,\;\beta\leq r,
\end{align}
where $s$ is the rank of $E$.
Clearly, there exists $c$ such that $\left(H+cH'\right)|_{F}$ is strictly positive. Thus, without loss of generality, we can assume $H|_{F}=0$.
Then for any $$U=\sum\limits_{i=1}^sU^ie_i,\,V=\sum\limits_{\alpha=s+1}^rV^{\alpha}e_{\alpha},$$ we may prove
\begin{align*}
&H(U+tV,U+tV)+cH'(X+tV,U+tV)  \\
=&H(U,U)+2t\mathrm{Re}(H(U,V))+t^2cH'(V,V)>0
\end{align*}
for any $t$, which is equivalent to
\begin{align*}
\left[\mathrm{Re}(H(U,V))\right]^2<cH'(V,V)H(U,U)=c\left(\sum\limits_{i=1}^s|U^i|^2\right)
\left(\sum\limits_{\alpha=s+1}^r|V^{\alpha}|^2\right).
\end{align*}
On the other hand, we have
\begin{align*}
\left[\mathrm{Re}(H(U,V))\right]^2
=&\left[\sum\limits_{1\leq i\leq s,s+1\leq \alpha\leq r}\mathrm{Re}(H_{i\alpha}U^i\overline{V}^{\alpha})\right]^2\\
\leq&\left(\sum\limits_{1\leq i\leq s,s+1\leq \alpha\leq r}|H_{i\alpha}|^2\right)\left(\sum\limits_{i=1}^s|U^i|^2\right)\left(\sum\limits_{\alpha=s+1}^r|V^{\alpha}|^2\right).
\end{align*}
The summation $\left(\sum\limits_{1\leq i\leq s,s+1\leq \alpha\leq r}|H_{i\alpha}|^2\right)$ can be locally bounded by $c$, hence $H+cH'$ is positive.
\end{proof}
\begin{thm}
\label{thmbalance}
 Suppose that $h:\;Q\longrightarrow \mathbb{P}^1$ is a holomorphic map, where $Q$ is a compact complex manifold with balanced metric $\omega_{Q}$ and $\dim_{\mathbb{C}}Q=n$. Let $M$ be a compact hypercomplex manifold with $\dim_{\mathbb{C}}M=r$ and $\omega_M$ defined as in \eqref{omegam}. Then there exists a balanced metric on $(X,\,\underline{\mathbf{I}})$.
\end{thm}
\begin{proof}
Note that
\begin{equation*}
\Lambda^{n+r-1}T^{1,0}X=\left(\Lambda^{r-1}T^{1,0}M\otimes\Lambda^{n}T^{1,0}Q\right)
+\left(\Lambda^{r}T^{1,0}M\otimes\Lambda^{n-1}T^{1,0}Q\right)=:E\oplus F.
\end{equation*}
Using \eqref{0120q} and \eqref{1002q}, we get
\begin{align*}
\md \omega_{M}^r=\md_M\omega_M^r+\partial_Q\omega_M^r+\overline{\partial}_Q\omega_M^r
=r\partial_Q\omega_M\wedge\omega_{M}^{r-1}+r\overline{\partial}_Q\omega_M\wedge\omega_{M}^{r-1}=0.
\end{align*}
This together with the fact that $\md \omega_{Q}^{r-1}=0$ implies that
$
\omega_{M}^r\wedge\omega_{Q}^{n-1}
$
is a  closed $(n+r-1,n+r-1)$ form and positive on $F$.  Furthermore, we get
\begin{align*}
E\subset \mathrm{Ker}\left(\omega_{M}^r\wedge\omega_{Q}^{n-1}\right).
\end{align*}
Denote by
$$
A:=\Big\{q:\;\md h_{q}=0\Big\}
$$
the set of critical point of $h$ which is analytic set. Indeed, for any $q\in A$, there exists a coordinate chart $(U;\,z^1,\cdots,z^n)$ such that $$A\cap U=\left\{q\in U:\;\frac{\partial h}{\partial z^1}=\cdots=\frac{\partial h}{\partial z^n}=0\right\}.$$
For any $q\in Q\backslash A$, we have $h^{\ast}\omega_{\mathbb{P}^1}\neq0$ and hence $\left(h^{\ast}\omega_{\mathbb{P}^1}\right)\wedge\omega_{Q}^{n-1}$ is a closed positive $(n,n)$ form which is useful to construct a closed positive $(n+r-1,\,n+r-1)$ form later. However, for any $q\in A$, we have $h^{\ast}\omega_{\mathbb{P}^1}=0$ and hence we need some modification term to obtain a closed positive $(n+r-1,n+r-1)$ form. At this point, without loss of generality, we can choose a local chart $(U_q,\;z^1,\cdots,z^n)$ such that
$$A\cap U_q\subset \{z^1=0\}.$$
Take a cut-off function $\varphi\in C^{\infty}(Q,\mathbb{R})$ such that $\mathrm{supp}\varphi \subset U_q$ and $\varphi|_{V_q}\equiv 1$, where $V_q$ is another open neighborhood of $q$ with $\overline{V_q}\subset U_q$.
Then we can define a closed form
$$
\xi=\ddbar\left((1+|z^1|^2)^{t}\varphi\omega_{M}^{r-1}\right)
$$
with $t$ some positive constant. On the set $\{z^1=0\}$, we have
\begin{align}
\label{localrevise}
\xi=&\mn\varphi t\md z^1\wedge\md \overline{z}^1\wedge\omega_{M}^{r-1}
+\mn\varphi\partial\overline{\partial}\omega_{M}^{r-1}\\
&+\mn\partial_{Q}\varphi\wedge\overline{\partial}\omega_{M}^{r-1}
-\mn\overline{\partial}_{Q}\varphi\wedge\partial\omega_{M}^{r-1}\nonumber\\
=&\mn t\omega_{M}^{r-1}\wedge\md z^1\wedge\md \overline{z}^1
+(-1)^{\delta_{2,r}}2(r-1)\omega_{M}^{r-1}\wedge h^{\ast}\omega_{\mathbb{P}^1}
+\sqrt{-1}\partial_{M}\overline{\partial }_{M}\omega _{M}^{n-1}\nonumber\\
&+\mn\partial_{Q}\varphi\wedge\overline{\partial}\omega_{M}^{r-1}
-\mn\overline{\partial}_{Q}\varphi\wedge\partial\omega_{M}^{r-1}.\nonumber
\end{align}
By Lemma \ref{basiclem}, if follows that
\begin{align}
\label{cutoff1}
\mn\partial_{Q}\varphi\wedge\overline{\partial}\omega_{M}^{r-1}\wedge\omega_{Q}^{n-1}
\in \Lambda^{r,r-2}M\oplus\Lambda^{n,n}Q
+\Lambda^{r-1,r}M\oplus\Lambda^{n,n-1}Q\\
\label{cutoff2}
\mn\overline{\partial}_{Q}\varphi\wedge\partial\omega_{M}^{r-1}\wedge\omega_{Q}^{n-1}
\in\Lambda^{r-2,r}M\oplus\Lambda^{n,n}Q
+\Lambda^{r,r-1}M\oplus\Lambda^{n-1,n}Q
\end{align}
From \eqref{localrevise}, \eqref{cutoff1} and \eqref{cutoff2}, we can deduce that  on $A\cap U_q$
$$
\left.\xi\wedge \omega_{Q}^{n-1}\right|_{E}=\mn \varphi t\omega_{M}^{r-1}\wedge\md z^1\wedge\md \overline{z}^1\wedge \omega_{Q}^{n-1}
$$
is non-negative, and that on $A\cap V_q$,
$$
\left.\xi\wedge \omega_{Q}^{n-1}\right|_{E}=\mn t\omega_{M}^{r-1}\wedge\md z^1\wedge\md \overline{z}^1\wedge \omega_{Q}^{n-1}
$$
is positive.
Since $Q$ is compact, we can choose finite such $\xi$ with these $V_q$'s covering $A$, denoted by $\xi_1,\cdots,\xi_{\ell}$, to obtain that
$$
\left.\left(t'(-1)^{\delta_{2,r}}\ddbar\omega_{M}^{r-1}
+\sum\limits_{i=1}^{\ell}\xi_i\right)\wedge\omega_{Q}^{r-1}\right|_{E}
$$
is positive with sufficiently large $t'$. Applying Lemma \ref{lem2} to $\Lambda^{n+r-1}T^{1,0}X$, $\omega_{M}^r\wedge\omega_{Q}^{n-1} $ and the $(n+r-1,\,n+r-1)$ form obtained right now,
it follows that there exists a sufficiently large number $\gamma$ such that
$$
\Omega_{X}:=\gamma \omega_{M}^r\wedge\omega_{Q}^{n-1}
+\left(t'(-1)^{\delta_{2,r}}\ddbar\omega_{M}^{r-1}+\sum\limits_{i=1}^{\ell}\xi_i\right)\wedge\omega_{Q}^{n-1}
$$
is a closed positive $(n+r-1,n+r-1)$ form. Then Lemma \ref{11n1n1} implies that there exists a balanced metric on $(X,\underline{\mathbf{I}})$.
\end{proof}
\begin{cor}
Under the same setup as in Theorem \ref{thmbalance}, and furthermore assume that $h$ is holomorphic submersion. Then there exists a balanced metric on $(X,\,\underline{\mathbf{I}})$.
\end{cor}
\begin{proof}
In this setup, $A=\emptyset$, hence the conclusion is obvious and we have
\begin{equation}
\Omega_{X}=\gamma \omega_{M}^r\wedge\omega_{Q}^{n-1}
+(-1)^{\delta_{2,r}}\ddbar\omega_{M}^{r-1} \wedge\omega_{Q}^{n-1}.
\end{equation}
Moreover, we can give the concrete expression of the balanced metric $\omega_X$ using \eqref{n-1n-111formulayong}.
To see this, we consider the Chern connection on $(M,\mathbf{I}_M,\,\omega_M)$. Using  the local coordinates $(U,w^1,\cdots, w^r)$, we write
$$\omega_M=\mn g_{i\overline{j}}\md w^i\wedge\md \overline{w}^j.$$
Note that for any real $(1,1)$ form $\chi=\mn \chi_{i\overline{j}}\md w^i\wedge\md \overline{w}^j$, we have
\begin{equation}
\label{traceformula}
\left(\tr_{\omega_M}\chi\right)\omega_M^r:=\left(g^{\overline{j}i}\chi_{i\overline{j}}\right)\omega_M^r=r\chi\wedge\omega_M^{r-1}.
\end{equation}
Then the Christoffel symbol of the Chern connection is $\Gamma_{ij}^k=g^{\overline{q}k}\partial_ig_{j\overline{q}}$, where and henceforth we denote by $\partial_i$ the partial derivative $\partial/\partial w^i$.

We define the torsion of the Chern connection by
$$
\quad T_{ij}^k=\Gamma_{ij}^k-\Gamma_{j i}^k=g^{\overline{q}k}\left(\partial_ig_{j\overline{q}}
-\partial_jg_{i\overline{q}}\right).
$$
and often lower its upper index using $\omega_M$, writing
$$
T_{ij\overline{\ell}}:=T_{ij}^kg_{k\overline{\ell}}=\partial_ig_{j\overline{\ell}}
-\partial_jg_{i\overline{\ell}}.
$$
Then we get
\begin{align}
\label{partialomega}
\overline{\partial}_M\omega
=&\mn \partial_{\overline{\ell}}g_{i\overline{j}}\md w^i\wedge\md \overline{w}^j\wedge\md \overline{w}^{ \ell }\\
=&\frac{\mn}{2}\left(\partial_{\overline{\ell}}g_{i\overline{j}}-\partial_{\overline{j}}g_{i\overline{\ell}}\right)\md w^i\wedge\md \overline{w}^j\wedge\md \overline{w}^{ \ell }\nonumber\\
=&\frac{\mn}{2}\overline{T_{\ell j\overline{i}}}\md w^i\wedge\md \overline{w}^j\wedge\md \overline{w}^{ \ell }\nonumber
\end{align}
Similarly, we have
\begin{align}
\label{partialbaromega}
\partial_M\omega=\frac{\mn}{2}T_{ij\overline{\ell}}\md w^i\wedge\md w^j\wedge\md \overline{w}^{\ell}
\end{align}
and
\begin{align}
\label{ddbaromega}
\mn\partial_M\overline{\partial}_M\omega=\frac{(\mn)^2}{2!2!}\left(\partial_{\overline{k}}T_{ij\overline{\ell}}-\partial_{\overline{\ell}}T_{ij\overline{k}}\right)
\md w^i\wedge\md w^j\wedge\md\overline{w}^k\wedge\md \overline{w}^{\ell}.
\end{align}
Moreover, using \eqref{partialomega} and \eqref{partialbaromega}, for $r>2$, we can deduce
\begin{align}
\label{partialomegadeng}
\mn\partial_M\omega\wedge\dbar_M\omega\wedge\omega^{r-3}=\frac{\Psi}{r(r-1)(r-2)} \omega^r
\end{align}
where
$$
\Psi=g^{\overline{j}i}g^{\overline{q}p}g^{\overline{\ell}k}T_{ip\overline{\ell}}\overline{T_{jq\overline{k}}}
-T_{ip}^p\overline{T_{jq}^q}g^{\overline{j}i}.
$$
Thanks to \eqref{ddbaromega}, it follows that
\begin{equation}
\label{ddbaromegadeng}
\mn\partial_M\overline{\partial}_M\omega\wedge\omega^{r-2}=\frac{\Phi}{r(r-1)}\omega^r,
\end{equation}
where
$$
\Phi=g^{\overline{\ell}i}g^{\overline{k}j}\left(\partial_{\overline{k}}T_{ij\overline{\ell}}-\partial_{\overline{\ell}}T_{ij\overline{k}}\right)
-g^{\overline{\ell}j}g^{\overline{k}i}\left(\partial_{\overline{k}}T_{ij\overline{\ell}}-\partial_{\overline{\ell}}T_{ij\overline{k}}\right).
$$
By \eqref{partialomegadeng}, \eqref{ddbaromegadeng} and Lemma \ref{basiclem}, it follows that
$$
\ddbar\omega_{M}^{r-1}=\frac{\Phi+\frac{1+(-1)^{\delta_{2,r}}}{2}\Psi}{r}\omega_{M}^r
+2(-1)^{\delta_{2,r}}(r-1)\left(h^{\ast}\omega_{\mathbb{P}^1}\right)\wedge\omega_{M}^{r-1}.
$$
Hence we have
\begin{align*}
 \Omega_X=&\gamma \omega_M^r\wedge\omega_{Q}^{n-1}+ (-1)^{\delta_{2,r}}\ddbar(\omega_{M}^{r-1})\wedge\omega_{Q}^{n-1}\\
 =&\left(\gamma+\frac{\Phi+\frac{1+(-1)^{\delta_{2,r}}}{2}\Psi}{r}\right)\omega_{M}^r
 \wedge\omega_{Q}^{n-1}+\frac{2(r-1)\tr_{\omega_Q}\left(h^{\ast}\omega_{\mathbb{P}^1}\right)}{r}\omega_{M}^{r-1}
 \wedge \omega_{Q}^n,
\end{align*}
where we use the analog of \eqref{traceformula} on $Q$.

Using \eqref{n-1n-111formulayong}, we get
\begin{align*}
\omega_X
 =&\left[\left(\gamma+\frac{\Phi+\frac{1+(-1)^{\delta_{2,r}}}{2}\Psi}{r}\right)(n-1)!r!\right] ^{\frac{n}{n+r-1}}\left(\frac{2(r-1)\tr_{\omega_Q}\left(h^{\ast}\omega_{\mathbb{P}^1}\right)}{r}  (r-1)!n!\right)^{-\frac{n-1}{n+r-1}}\omega_{M}  \\
 &+\left[\left(\gamma+\frac{\Phi+\frac{1+(-1)^{\delta_{2,r}}}{2}\Psi}{r}\right)(n-1)!r!\right]  ^{-\frac{r-1}{n+r-1}}\left(\frac{2(r-1)\tr_{\omega_Q}\left(h^{\ast}\omega_{\mathbb{P}^1}\right)}{r}   (r-1)!n!\right)^{\frac{r}{n+r-1}} \omega_{Q},
\end{align*}
with $\frac{\omega_{X}^{n+r-1}}{(n+r-1)!}=\Omega_X$.
\end{proof}
Let $\bar{h}:\,\bar Q\longrightarrow Q$ be a branched double cover along the smooth divisor $S$. Then by the new holomorphic map  $h\circ \bar h:\,\bar Q\longrightarrow \mathbb{P}^1$, we get a new complex manifold $\bar X=M\times \bar Q$. This can be seen as a branched double cover of $X$ along the divisor $\pi^{-1}(S)$. By \cite[Proposition 4.1.6]{positivity1}, $\bar h^{-1}(S)$ is also a smooth divisor. Theorem \ref{thmbalance} implies that $\bar X$ is also balanced.

Let $(M,\,\mathbf{I},\mathbf{J},\mathbf{K},g)$ be a hypercomplex manifold, not necessarily compact, with holonomy group $\mathrm{Hol}(\nabla)\subset SL(n,\mathbb{H})$, where $\nabla$ is the Obata connection, or a hyper-K\"{a}hler manifold, not necessarily compact.  Then there exists a countable set $B\subset \mathbb{S}^2$ biholomorphic to $\mathbb{P}^1$, such that for any $(a,b,c)\in \mathbb{S}^2 \setminus B$, we can deduce that $(M, a\mathbf{I}+b\mathbf{J}+c\mathbf{K})$ has no compact divisors (see \cite{sove,verbitsky}).  Therefore, let $X$ be the higher dimensional analogs of twistor spaces constructed from compact hyper-K\"{a}hler manifold or compact hypercomplex manifold $M$ with holonomy group $\mathrm{Hol}(\nabla)\subset SL(n,\mathbb{H})$, and let $D\subset X$ be a smooth divisor and hence a complex sub-manifold. If $(\tilde{\varsigma}\circ h\circ \pi)|_{D}$ is a non-constant map, then it is a surjective holomorphic map and for any $(a,b,c)\in \mathbb{S}^2 \setminus B$, we know that $D\cap \left(\pi^{-1}\circ h^{-1}\circ \tilde{\varsigma}^{-1}(a,b,c)\right)$ is $(M, a\mathbf{I}+b\mathbf{J}+c\mathbf{K})$ itself  or that it satisfies $\mathrm{codim}\Big(D\cap \left(\pi^{-1}\circ h^{-1}\circ \tilde{\varsigma}^{-1}(a,b,c)\right)\Big) \leq 2$, and the latter case is impossible since we have $\mathrm{codim}D=1$. The same conclusion holds when $(\tilde{\varsigma}\circ h\circ \pi)|_{D}$ is a constant map and $(\tilde{\varsigma}\circ h\circ \pi) (D)\subset \mathbb{S}^2 \setminus B$. From this point, our choices of the divisors to construct branched double covers are not so limited.
\subsection{Non-K\"{a}hlerity}
\label{sec_non-Kahlerity}
In this subsection, we show that the higher dimensional analogs $X$ are not K\"ahler.
\begin{thm}\label{nonkahler}
 Suppose that $h:\;Q\longrightarrow \mathbb{P}^1$ is a holomorphic map, where $Q$ is a compact complex manifold with balanced metric $\omega_{Q}$ and $\dim_{\mathbb{C}}Q=n$. Let $M$ be a compact hypercomplex manifold with $\dim_{\mathbb{C}}M=r$ and $\omega_M$ defined as in \eqref{omegam}. Then  $(X,\,\underline{\mathbf{I}})$ can not be K\"{a}hlerian.
\end{thm}
\begin{proof}
We use proof by contradiction. Assume that $X$ is K\"{a}hlerian and $\omega_{X}$ is the K\"{a}hler form on it. Since for every $q\in Q$ fixed, $M$ can be seen as a fiber $\pi^{-1}(q)$ of $X$, if follows that $M$ is K\"{a}herian, and hence is hyperK\"{a}hlerian by the result in \cite{vergafa}. Without loss of generality, we can still assume that $\omega_M$ is a K\"{a}hler form on it, and hence \eqref{1111q} implies
$$
\ddbar\omega_M=\sqrt{-1}\partial_{Q} \overline{\partial}_{Q}\omega_M=-2 \left(h^{\ast}\omega_{\mathbb{P}^1}\right)\wedge \omega_M.
$$
Then we can deduce that
$$
-\left(\ddbar\omega_M\right)\wedge \omega_{X}^{n+r-2}
$$
is a real nonnegative $(n+r,\,n+r)$ form on $X$ and is a strictly positive $(n+r,\,n+r)$ form on $X\backslash\pi^{-1}(A)$, where $A$ is defined as in the proof of Theorem \ref{thmbalance}, and $\pi^{-1}(A)$ has zero $(n+r)$ measure. Thus, using Stokes' theorem, we have
$$
0<\int_{X} -\left(\ddbar\omega_M\right)\wedge \omega_{X}^{n+r-2}=-\int_{X}\md\left[\left( \mn\dbar \omega_M\right)\wedge \omega_{X}^{n+r-2}\right]=0,
$$
which leads to a contradiction. Hence the proof is completed.
\end{proof}
The proof of Theorem \ref{nonkahler} yields that
\begin{thm}
Suppose that $h:\;Q\longrightarrow \mathbb{P}^1$ is a holomorphic map, where $Q$ is a compact complex manifold with balanced metric $\omega_{Q}$ and $\dim_{\mathbb{C}}Q=n$. Let $M$ be a compact hyper-K\"{a}hler manifold with $\dim_{\mathbb{C}}M=r$ and $\omega_M$ defined as in \eqref{omegam}. Then  $(X,\,\underline{\mathbf{I}})$ can not be K\"{a}hlerian.
\end{thm}

\section{\textbf{Discussion}}

\label{sec_discussion}

By generalizing the twistor $\mathbb{P}^{1}$ to a more general complex manifold $Q$, we constructed a generalization of twistor spaces of hypercomplex manifolds and hyper-K\"{a}hler manifolds $M$. We found that the manifold $X$ constructed in this way is complex if and only if $Q$ admits a holomorphic map to $\mathbb{P}^1$.

We showed that these manifolds and their branched double covers are complex non-K\"{a}hler. We made branched double covers of these manifolds, branching along appropriate divisors. Some of these branched double covers can provide non-K\"{a}hler Calabi-Yau manifolds. If in addition $Q$ is a balanced manifold, the resulting manifold $X$ and its special double cover have balanced Hermitian metrics. We found their explicit Hermitian metrics by methods of positivity.

It may be possible to make blowing-downs of these manifolds, under which
they could become projective. In the context of the twistor spaces for self-dual
manifolds and their branched double covers, these blowing-downs can be
performed, see for example \cite{Poon86, Lebrun Poon, Lin:2014lya}.
Moreover, these geometries can be interesting in understanding the moduli
space of Calabi-Yau manifolds \cite{Reid}.

One can also construct stable vector bundles on them \cite{Donaldson,UY, LY
hermitian}. They are also interesting in the context of string theory. The
balanced manifolds constructed in this paper would be useful for the exploration
of stable vector bundles on them. The existence of the solution to the
Hermitian Yang-Mills equations on these manifolds is expected to be
equivalent to the stability of the vector bundle on them.

The non-K\"{a}hler geometries considered here could be useful for mirror symmetry \cite{Strominger:1996it} in higher dimensions and in non-K\"{a}hler manifolds \cite{Lau:2014fia,Minasian:2016txd}. It may be interesting to identify a subclass of these manifolds in this construction that will be useful for the mirror symmetry.

\end{document}